\documentclass{amsart}
\usepackage{amsmath,amssymb,amsthm}
\usepackage{MnSymbol,mathrsfs}
\usepackage{enumitem}

\newcommand{\pow}{\mathscr{P}} 
\newcommand{\filt}{\mathcal{F}} 
\newcommand{\csf}{{2}\sp\omega} 
\newcommand{\pair}[1]{\langle #1\rangle} 
\newcommand{\cube}{\mathcal{Q}} 
\newcommand{\cubein}{\mathsf{s}} 
\newcommand{\id}{\mathbf{1}} 

\newtheorem{thm}{Theorem}[section]
\newtheorem{lemma}[thm]{Lemma}
\newtheorem{coro}[thm]{Corolary}
\newtheorem{ques}[thm]{Question}

\title{A non-discrete space $X$ with $C_p(X)$ Menger at infinity}
\author[Bella]{Angelo Bella}
\author[Hern\'andez]{Rodrigo Hern\'andez-Guti\'errez}

\address[Bella]{Department of Mathematics and Computer Science, University of Catania, Citt\'a universitaria, viale A. Doria 6, 95125 Catania, Italy}
\email[Bella]{bella@dmi.unict.it}
\address[Hern\'andez]{Departamento de Matem\'aticas, Universidad Aut\'onoma Metropolitana campus Iztapalapa, Av. San Rafael Atlixco 186, Col. Vicentina, Iztapalapa, 09340, Mexico city, Mexico}
\email[Hern\'andez]{rodrigo.hdz@gmail.com}

\date{\today}


\begin{document}

\begin{abstract}
In a paper by Bella, Tokg\"os and Zdomskyy it is asked whether there exists a Tychonoff space $X$ such that the remainder of $C_p(X)$ in some compactification is Menger but not $\sigma$-compact. In this paper we prove that it is consistent that such space exists and in particular its existence follows from the existence of a Menger ultrafilter.
\end{abstract}

\maketitle

\section{Introduction}

A space $X$ is called Menger if for every sequence $\{\mathcal{U}_n:n\in\omega\}$ of open covers of $X$ one may choose finite sets $\mathcal{V}_n\subset\mathcal{U}_n$ for all $n\in\omega$ in such a way that $\bigcup\{\mathcal{V}_n:n\in\omega\}$ covers $X$. Given a property $\mathbf{P}$, a Tychonoff space $X$ will be called \emph{$\mathbf{P}$ at infinity} if $\beta X\setminus X$ has $\mathbf{P}$.
 
Let $X$ be a Tychonoff space. It is well-known that $X$ is $\sigma$-compact at infinity if and only if $X$ is \v Cech-complete. Also, Henriksen and Isbell in \cite{hen-is} proved that  $X$ is Lindel\"of at infinity if and only if $X$ is of countable type. Moreover, the Menger property implies the Lindel\"of property and is implied by $\sigma$-compactness. So it was natural for the authors of \cite{aurichi-bella} to study when  $X$ is Menger at infinity.

Later, the authors of \cite{menger-remainders} study when a topological group is Menger, Hurewicz and Scheepers at infinity. The Hurewicz and Scheepers properties are other covering properties that are stronger than the Menger property and weaker than $\sigma$-compactness (see the survey \cite{tsaban-survey} by Boaz Tsaban). Essentially, \cite{menger-remainders} has two main results. 

\begin{thm}\cite[Theorem 1.3]{menger-remainders}
If $G$ is a topological group and $\beta G\setminus G$ is Hurewicz, then $\beta G\setminus G$ is $\sigma$-compact.
\end{thm} 

\begin{thm}\cite[Theorem 1.4]{menger-remainders}
There exists a topological group $G$ such that $\beta G\setminus G$ is Scheepers and not $\sigma$-compact if and only if there exists an ultrafilter $\mathcal{U}$ on $\omega$ such that, considered as a subspace of $\pow(\omega)$ with the Cantor set topology, $\mathcal{U}$ is Scheepers.
\end{thm}

The last section of \cite{menger-remainders} considers the specific case of the topological group $C_p(X)$ consisting of all continuous real-valued functions defined on $X$, with the topology of pointwise convergence. It is shown that if $C_p(X)$ is Menger at infinity, then it is first countable and hereditarily Baire. It is a well-known result that $C_p(X)$ is \v Cech-complete (equivalently, $\sigma$-compact at infinity) if and only if $X$ is countable and discrete (see \cite[I.3.3]{arkh-top_func_spaces}). So the authors of \cite{menger-remainders} made the natural conjecture that $C_p(X)$ is Menger at infinity if and only if $C_p(X)$ is $\sigma$-compact at infinity (\cite[Question 6.1]{menger-remainders}). Their conjecture is equivalent to the statement that if $C_p(X)$ is Menger at infinity, then $X$ is countable and discrete. In this paper we disprove this conjecture.

\begin{thm}\label{main}
It is consistent with ZFC that there exists a regular, countable, non-discrete space $X$ such that $C_p(X)$ is Menger at infinity.
\end{thm}

Our proof of Theorem \ref{main} uses filters with a special property that is immediately satisfied by Menger ultrafilters. See Theorem \ref{actualthm} for the exact property we use.

According to \cite[Theorem 3.5]{chod-rep-zdo}, Menger filters are precisely those called Canjar filters. Also, by \cite[Proposition 2]{canjar1}, $\mathfrak{d}=\mathfrak{c}$ implies there is a Canjar (thus, Menger) ultrafilter. However, a characterization of Canjar ultrafilters given in \cite{canjar1} implies that a Menger ultrafilter is a $P$-point. Thus, Menger ultrafilters consistently do not exist.

 However, we do not know whether there are filters in ZFC that satisfy the conditions we need. See sections \ref{proofsection} and \ref{quessection}  for a more thorough explanation and concrete open questions.

\section{Preliminaries and notation}

\subsection{The Menger property}The Menger property has been thoroughly studied. We state some well-known facts below:
\begin{enumerate}[label={(\roman*)}]
\item \cite[Theorem 2.2]{COC2} Every $\sigma$-compact set is Menger.
\item \cite[Theorem 3.1]{COC2} If $X$ is Menger and $Y\subset X$ is closed, then $Y$ is Menger.
\item If $X$ is Menger and $K$ is $\sigma$-compact, then $X\times K$ is also Menger.
\item \cite[Theorem 3.1]{COC2} The continuous image of a Menger space is also Menger.
\item The continuous and perfect pre-image of a Menger space is also Menger.
\item \cite[p. 255]{COC2} If a space is the countable union of Menger spaces, then it is Menger as well.
\item $\omega\sp\omega$ is not Menger.
\end{enumerate}

We also mention the following observation of Aurichi and Bella.

\begin{lemma}\cite[Corollary 1.6]{aurichi-bella}\label{simplecriterion}
A space $X$ is Menger at infinity if and only if there exists a compactification of $X$ with a Menger remainder if and only if the remainder of every compactification of $X$ is Menger.
\end{lemma} 

\subsection{Filters} A filter $\filt$ on a non-empty set $X$ is a subset $\filt\subset\pow(X)$ such that: (a) $\emptyset\notin\filt$, (b) if $x,y\in\filt$ then $x\cap y\in\filt$, and (c) if $x\in\filt$ and $x\subset y\subset X$, then $y\in\filt$.  All filters in this paper are defined on countable sets (and most of the times, on $\omega$). Filters that contain the Fr\'echet filter of cofinite sets are called free. Maximal filters are called ultrafilters. Let $\chi:\pow(\omega)\to \csf$ be the function that sends each subset to its characteristic function. Using $\xi$, a filter on a countable set can be thought of as a subspace of the Cantor set.

For every subset $\mathcal{Y}\subset\pow(X)$ we may define $\mathcal{Y}\sp\ast=\{A\subset X: X\setminus A\in\mathcal{Y}\}$. If $\filt$ is a filter on $\omega$, $\filt\sp\ast$ is called its dual ideal and $\filt\sp+=\pow(X)\setminus\filt\sp\ast$ is the set of $\filt$-positive sets. Moreover, the function that takes each set in $\pow(X)$ to its complement is a homeomorphism. Thus, a filter is always homeomorphic to its dual ideal. Also, notice that the complement $\pow(X)\setminus\filt$ is then homeomorphic to $\filt\sp+$.

Given $A\subset\omega\times\omega$ and $n\in\omega$, define
$$
\begin{array}{lrl}
A\sp{(n)}&=&\{i\in\omega: \pair{i,n}\in A\},\textrm{ and}\\
A_{(n)}&=&\{i\in\omega: \pair{n,i}\in A\}.
\end{array}
$$
For $\mathcal{Y}\subset\pow(\omega)$, let us define
$$
\mathcal{Y}\sp{(\omega)}=\{A\subset\omega\times\omega:\forall n\in\omega\ (A_{(n)}\in\mathcal{Y}))\}.
$$
There is a natural function from $\pow(\omega\times\omega)$ to $\pow(\omega)\sp\omega$  that takes each $A\subset\omega\times\omega$ to $\{\pair{n,A_{(n)}}:n\in\omega\}$. This function is also a homeomorphism and takes $\mathcal{Y}\sp{(\omega)}$ to $\mathcal{Y}\sp\omega$.

It is easy to see that if $\filt$ is a filter on $\omega$, then $\filt\sp{(\omega)}$ is a filter on $\omega\times\omega$. Thus, the $\omega$-power of a filter is always (homeomorphic to) a filter.

A filter $\filt$ is a \emph{$P$-filter} if for every $\{F_n:n<\omega\}\subset\filt$ there exists $F\in\filt$ such that $F\setminus F_n$ is finite for all $n\in\omega$. A filter is called \emph{meager} if it is meager as a topological space. It is known that ultrafilters are non-meager \cite[Theorem 4.1.1]{bart}. A $P$-point is a $P$-filter that is also an ultrafilter. The existence of $P$-points is independent from ZFC. For example, $\mathfrak{d}=\mathfrak{c}$ implies there are $P$-points but there are models with no $P$-points, see \cite[section 4.4]{bart}.   Non-meager $P$-filters are a natural generalization of $P$-points; it is still an open question whether they exist in ZFC but if they don't exist, then there is an inner model with a large cardinal, see \cite[section 4.4.C]{bart}.

\subsection{The Hilbert cube} The \emph{Hilbert cube} is the countable infinite product of closed intervals of the reals, we will find it convenient to work with $\cube=[-1,1]\sp\omega$. The pseudointerior of $\cube$ is $\cubein=(-1,1)\sp\omega$ and the pseudoboundary is $B(\cube)=\cube\setminus \cubein$.

\section{The example}\label{proofsection}

According to \cite[Proposition 6.2]{menger-remainders}, if $C_p(X)$ is Menger at infinity, then it is first countable and hereditarily Baire. From \cite[I.1.1]{arkh-top_func_spaces}, it follows that $X$ is countable. In \cite{marciszewski}, Witold Marciszewski studied those countable spaces $X$ such that $C_p(X)$ is hereditarily Baire. We will consider one specific case: when $X$ has a unique non-isolated point. 

Given a filter $\filt\subset\pow(\omega)$, consider the space $\xi(\filt)=\omega\cup\{\filt\}$, where every point of $\omega$ is isolated and every neighborhood of $\filt$ is of the form $\{\filt\}\cup A$ with $A\in\filt$. All of our filters will be \emph{free}, that is, they contain the Fr\'echet filter. In this case, $\filt$ is not isolated. When $\filt$ is the Fr\'echet filter, $\xi(\filt)$ is homeomorphic to a convergent sequence. It is easy to see that a space $X$ is homeomorphic to a space of the form $\xi(\filt)$ if and only if $X$ is a countable space with a unique non-isolated point.

\begin{thm}\cite{marciszewski}
For a free filter $\filt$ on $\omega$, the following are equivalent.
\begin{enumerate}[label=(\alph*)]
\item $\filt$ is a non-meager $P$-filter,
\item $\filt$ is hereditarily Baire, and
\item $C_p(\xi(\filt))$ is hereditarily Baire.
\end{enumerate}
\end{thm}

Thus, it is natural to ask when $C_p(\xi(\filt))$ is Menger at infinity. By Lemma \ref{simplecriterion} it is sufficient to look at any compactification of $C_p(\xi(\filt))$ and try to decide whether the remainder is Menger.

Consider the set of functions in the Hilbert cube that $\filt$-converge to $0$:
$$
K_\filt=\{f\in\cube:\forall m\in\omega\ \{n\in\omega:\lvert f(n)\rvert<2\sp{-m}\}\in\filt\},
$$
and those that only take values in the pseudointerior
$$
C_\filt=K_\filt\cap\cubein.
$$

By \cite[Lemma 2.1]{marciszewski-analytic_coanalytic}, it easily follows that $C_\filt$ is homeomorphic to $C_p(\xi(\filt))$. Also, $C_\filt$ is dense in $\cube$. So our problem becomes equivalent to finding a filter $\filt$ such that $\cube\setminus C_\filt$ is Menger. In fact, we will prove the following characterization of those filters.

\begin{thm}\label{actualthm}
Let $\filt$ be a free filter on $\omega$. Then $C_p(\xi(\filt))$ is Menger at infinity if and only if $\filt\sp\plus$ is Menger.
\end{thm}

Recall that since $\filt\sp\plus$ is homeomorphic to $\pow(\omega)\setminus \filt$, the property in Theorem \ref{actualthm} is equivalent to saying that $\filt$ is Menger at infinity. So the problem is reduced to the existence of such filters. As discussed in the introduction, Menger ultrafilters have the desired properties. Indeed, an ultrafilter coincides with its set of positive sets. Thus, we conclude the following.

\begin{coro}
If $\filt$ is a Menger ultrafilter on $\omega$, then $C_p(\xi(\filt))$ is Menger at infinity.
\end{coro}

Now we proceed to the proof of Theorem \ref{actualthm}. The argument is based on two classical theorems that relate the Cantor set and the unit interval: $[0,1]$ has a subspace homeomorphic to $\csf$ and is a continuous image of $\csf$. We just have to take the filter into account and the proof will follow naturally.

A family of closed, non-empty subsets $\{J_s:s\in 2\sp{<\omega}\}$ of a space $X$ will be called a \emph{Cantor scheme} on $X$ if 
 \begin{enumerate}[label=(\roman*)]
  \item for every $s\in2\sp{<\omega}$, $J_s\supset J_{s\sp\frown0}\cup J_{s\sp\frown1}$, and
  \item for every $f\in2\sp\omega$, $J_f=\bigcap\{J_{f\restriction_n}:n<\omega\}$ is exactly one point.
 \end{enumerate}
Let $\id=\omega\times\{1\}$ and $\sigma=\{\id\!\!\restriction_n:n\in\omega\}$.
 
 \begin{lemma}\label{lemma1}
For every free filter $\filt$ on $\omega$ there is a closed embedding of $\filt\sp{(\omega)}$ into $C_p(\xi(\filt))$. 
\end{lemma}
\begin{proof}
As explained earlier we shall work on $C_\filt\subset\cube$ instead of the function space. Recursively, construct a Cantor scheme $\{J_s:s\in2\sp{<\omega}\}$ in the interval $(-1,1)$ such that:
\begin{enumerate}[label=(\roman*)]
\item $J_\emptyset=[-\frac{1}{2},0]$,
\item for every $s\in 2\sp{<\omega}$, $J_s$ is a non-degenerate closed interval of length $<2\sp{-\lvert s\rvert}$,
\item for every $s\in 2\sp{<\omega}$, $J_{s\sp\frown 0}\cap J_{s\sp\frown 1}=\emptyset$, and
\item for every $s\in\sigma$, $0\in J_s$.
\end{enumerate}

Now define the function $\varphi:\pow(\omega\times\omega)\to\cube$ such that for all $A\subset\omega\times\omega$ and $n\in\omega$, $\varphi(A)(n)$ is the unique point in $J_{\chi(A\sp{(n)})}$. So informally speaking, the $n$-th row of $A$ is used to define the value of the function $\varphi(A):\omega\to 2$ at $n$.

From standard arguments, it is easy to see that $\varphi$ is an embedding. Now we shall prove that $A\in\filt\sp{(\omega)}$ if and only if $\varphi(A)\in K_\filt$.

First, assume that $A\in\filt\sp{(\omega)}$, that is, $A_{(i)}\in\filt$ for all $i\in\omega$. Let $m\in\omega$. By the definition of $\varphi(A)$, for every $n\in\bigcap\{A_{(i)}:i\leq m\}$ we have that $\varphi(A)(n)\in J_{\id\!\restriction m}$. Since $J_{\id\!\restriction m}$ has diameter less than $2\sp{-m}$ and contains $0$, we obtain that
$$
\bigcap\{A_{(i)}:i\leq m\}\subset\{n\in\omega:\lvert \varphi(A)(n)\rvert<2\sp{-m}\}.
$$
Thus the set $\{n\in\omega:\lvert \varphi(A)(n)\rvert<2\sp{-m}\}$ is an element of $\filt$. Since this holds for every $m<\omega$, we obtain that $\varphi(A)$ is $\filt$-convergent to $0$.

Now assume that $\varphi(A)\in K_\filt$ and fix $m\in\omega$. Let $k\in\omega$ be such that the length of $J_{\varphi\!\restriction m}$ is greater than $2\sp{-k}$. By our hypothesis the set $\{n\in\omega:\lvert \varphi(A)(n)\rvert<2\sp{-k}\}$ is an element of $\filt$ and
$$
\{n\in\omega:\lvert \varphi(A)(n)\rvert<2\sp{-k}\}\subset\{n\in\omega:\varphi(A)(n)\in J_{\varphi\restriction m}\}
$$
so $\{n\in\omega:\varphi(A)(n)\in J_{\varphi\restriction m}\}\in\filt$. Finally, notice that  by the definition of $\varphi$
$$
\{n\in\omega:\varphi(A)(n)\in J_{\varphi\restriction m}\}\subset\{n\in\omega:\pair{m,n}\in A\}=A_{(m)},
$$
so we obtain that $A_{(m)}\in\filt$. Since this is true for all $m\in\omega$, we conclude that $A\in\filt\sp{(\omega)}$.

This concludes the proof that $A\in\filt\sp{(\omega)}$ if and only if $\varphi(A)\in K_\filt$. Also, notice that the image of $\pow(\omega\times\omega)$ under $\varphi$ is a subset of $\cubein$. Thus, we can even say that $A\in\filt\sp{(\omega)}$ if and only if $\varphi(A)\in C_\filt$. Thus, $\varphi\!\!\restriction\filt\sp{(\omega)}$ is the closed embedding we wanted. 
\end{proof}

\begin{lemma}\label{lemma2}
Let $\filt$ be a free filter on $\omega$. Then there exists a continuous surjective function $\varphi:\pow(\omega\times\omega)\to\cube$ such that $\varphi\sp\leftarrow[K_\filt]=\filt\sp{(\omega)}$.
\end{lemma}
\begin{proof}
As before, we work on $C_\filt\subset\cube$ instead of the function space. Recursively, construct a Cantor scheme $\{J_s:s\in2\sp{<\omega}\}$ in the interval $[-1,1]$ according to the following conditions:
\begin{enumerate}
  \item $J_\emptyset=[-1,1]$,
  \item for every $s\in 2\sp{<\omega}$, $J_s$ is a non-degenerate closed interval of length $\leq\sp{2-\lvert s\rvert}$,
  \item if $s\in\sigma$, there are $0<x<y$ with $J_s=[-y,y]$, $J_{s\sp\frown1}=[-x,x]$ and $A_{s\sp\frown0}=[-y,-x]\cup[x,y]$,
  \item if $s\in\sigma$, there are $0<x<y$ with $J_{s\sp\frown\pair{0,0}}=[-y,-x]$ and $J_{s\sp\frown\pair{0,1}}=[x,y]$,
  \item if $s\in2\sp{<\omega}\setminus\sigma$ and there are $a<b$ with $J_s=[a,b]$, then there exists $x\in(a,b)$ such that $J_{s\sp\frown0}=[a,x]$ and $J_{s\sp\frown1}=[x,b]$.
 \end{enumerate}
Again define the function $\varphi:\pow(\omega\times\omega)\to\cube$ such that for all $A\subset\omega\times\omega$ and $n\in\omega$, $\varphi(A)(n)$ is the unique point in $J_{\chi(A\sp{(n)})}$.

Another standard argument implies that $\varphi$ is a continuous, surjective function. Also, the equality $\varphi\sp\leftarrow[K_\filt]=\filt\sp{(\omega)}$ can be proved in a manner completely analogous to the corresponding equality from Lemma \ref{lemma1}. Thus, we will leave this argument to the reader.
\end{proof}

Finally, the following allows us to simplify the characterization we will obtain.

\begin{lemma}\label{lemma3}
Let $\filt$ be a free filter. Then $\left(\filt\sp{(\omega)}\right)\sp\plus$ is Menger if and only if $\filt\sp+$ is Menger.
\end{lemma}
\begin{proof}
First, assume that $\filt\sp+$ is Menger. For each $n\in\omega$, consider $M_n=\{A\subset\omega\times\omega:A_{(n)}\in\filt\sp+\}$, which is homeomorphic to the product $\filt\sp+\times\pow(\omega)\sp\omega$. Since the product of a Menger space and a compact space is Menger, it follows that $M_n$ is Menger for every $n\in\omega$. Then notice that $(\filt\sp{(\omega)})\sp+=\bigcup\{M_n:n\in\omega\}$ is a countable union of Menger spaces so it is Menger.

Now assume that $(\filt\sp{(\omega)})\sp+$ is Menger. The diagonal in a product is always a closed subspace and the diagonal of $(\filt\sp{(\omega)})\sp+$ is equal to the set 
$$
\{A\subset\omega\times\omega:\exists F\in\filt\sp+\ \forall n\in\omega\ (A_{(n)}=F)\},
$$
which is homeomorphic to $\filt\sp+$. Then $\filt\sp+$ is Menger because it is a closed subspace of a Menger space.
\end{proof}

\begin{proof}[Proof of Theorem \ref{actualthm}]
By Lemma \ref{lemma3}, it is enough to prove that $C_\filt$ is Menger at infinity if and only if $\filt\sp{(\omega)}$ is Menger at infinity. 

First, assume that $C_\filt$ is Menger at infinity. This means that $\cube\setminus C_\filt$ is Menger. By Lemma \ref{lemma1}, $\filt\sp{(\omega)}$ can be embedded as a closed set $F$ in $C_\filt$. Let $\overline{F}$ denote the closure of $F$ in $\cube$. Then $\overline{F}\setminus F$ is a closed subset of $\cube\setminus C_\filt$. Then $\overline{F}$ is a compactification of $\filt\sp{(\omega)}$ with Menger remainder.

Now, assume that $\filt\sp{(\omega)}$ is Menger at infinity. Let $\varphi:\pow(\omega\times\omega)\to\cube$ be the continuous surjection from Lemma \ref{lemma2}. Then it follows that $\varphi[\pow(\omega\times\omega)\setminus\filt\sp{(\omega)}]=\cube\setminus K_\filt$. Since the Menger property is preserved under continuous functions, $\cube\setminus K_\filt$ is Menger. Notice that
$$
\cube\setminus C_\filt=(\cube\setminus K_\filt)\cup B(\cube).
$$
Since the boundary $B(\cube)$ is $\sigma$-compact and the union of countably many Menger spaces is Menger, $\cube\setminus C_\filt$ is Menger. This concludes the proof of the Theorem.
\end{proof}

\section{Questions}\label{quessection}

Let $\filt$ be a free filter on $\omega$ such that $\filt\sp+$ is Menger. We have just proved that $C_p(\xi(\filt))$ is Menger at infinity. By the observations of Aurichi and Bella from \cite{aurichi-bella}, $C_p(\xi(\filt))$ is a hereditarily Baire space. Then, by Marciszewski's result from \cite{marciszewski} it follows that $\filt$ is a non-meager $P$-filter. Thus, inadvertently we proved the following, which supersedes \cite[Observation 3.4]{menger-remainders} (for filters only).

\begin{coro}
Let $\filt$ be a free filter on $\omega$. If $\filt\sp+$ is Menger, then $\filt$ is a non-meager $P$-filter.
\end{coro}

Here is another more direct proof: Assume that $\filt$ is a free filter that is not a non-meager $P$-filter. By Lemmas 2.1 and 2.2 from \cite{marciszewski}, we obtain that $\filt$ has a closed subset $Q$ homeomorphic to the rationals. The closure $\overline{Q}$ of $Q$ in $\pow(\omega)$ must be homeomorphic to the Cantor set. Also, $\overline{Q}\setminus Q$ is homeomorphic to $\omega\sp\omega$, contained in $\pow(\omega)\setminus\filt$ and closed in $\pow(\omega)\setminus\filt$. Since $\omega\sp\omega$ is not Menger and the Menger property is hereditary to closed sets, $\pow(\omega)\setminus \filt$ is not Menger.

By \cite[2.7]{hg-sz} every filter of character $<\mathfrak{d}$ is a Menger filter. However, it not hard to conclude from \cite[4.1.2]{bart} that any filter of character $<\mathfrak{d}$ is meager. So indeed none of these filters can have its positive set Menger.

As discussed earlier, the existence of a non-meager $P$-filter in ZFC is still an open question. But so far the only example of a filter $\filt$ with $\filt\sp+$ Menger is a Menger ultrafilter, which we know that consistently does not exist. So it is natural to ask about the consistency of filters that are Menger at infinity.

\begin{ques}
Does ZFC imply that there exists a free filter $\filt$ on $\omega$ such that $\filt\sp+$ is Menger?
\end{ques}

\begin{ques}
Is the existence of a free Menger ultrafilter on $\omega$ equivalent to the existence of a free filter $\filt$ on $\omega$ such that $\filt\sp+$ is Menger?
\end{ques}

Finally, regarding the original question from \cite{menger-remainders}, we could ask for examples with other properties. In \cite[Proposition 3.3]{marciszewski}, Marciszewski gave general conditions for a countable space $X$ to have $C_p(X)$ hereditarily Baire. So we may ask what happens in general.

\begin{ques}
Does there exist a countable, regular, crowded space $X$ such that $C_p(X)$ is Menger at infinity?
\end{ques}

\begin{ques}
Does there exist a countable, regular, maximal space $X$ such that $C_p(X)$ is Menger at infinity?
\end{ques}

\begin{ques}
Characterize all countable regular spaces $X$ such that $C_p(X)$ is Menger at infinity.
\end{ques}

\section*{Acknowledgements}

The research that led to the present paper was partially supported by a grant of the group GNSAGA of INdAM. The second-named author was also supported by the 2017 PRODEP grant UAM-PTC-636 awarded by the Mexican Secretariat of Public Education (SEP).

\end{document}